\documentclass[10pt]{article}

\usepackage{amsfonts}
\usepackage{amsmath}
\usepackage{amssymb}
\usepackage{amsthm}
\usepackage{enumerate}
\usepackage{mdwlist}
\usepackage{mathrsfs}
\usepackage{bussproofs}
\usepackage{graphicx}
\usepackage{latexsym}
\usepackage{titlesec}
\usepackage{stmaryrd}
\usepackage{verbatim}
\usepackage{tcolorbox}
\tcbset{colback=white}

\newtheorem{theorem}{Theorem}[section]

\newtheorem{lemma}[theorem]{Lemma}
\newtheorem{corollary}[theorem]{Corollary}
\newtheorem{notation}[theorem]{Notation}

\theoremstyle{remark}

\newcommand{\RR}{\ensuremath{\mathbb R}}

\newcommand{\NN}{\ensuremath{\mathbb N}}

\newcommand{\norm}[1]{\left\| {#1}\right\|}
\newcommand{\fix}[1]{\mathrm{Fix}({#1})}
\newcommand{\inter}[1]{\mathrm{Int}({#1})}
\newcommand{\diam}{\mathrm{diam}}
\newcommand{\ceil}[1]{\lceil {#1} \rceil}
\newcommand{\asym}{asymptotically decreasing}
\newcommand{\seg}[1]{\mathrm{seg}[{#1}]}

\title{A new metastable convergence criterion and an application in the theory of uniformly convex Banach spaces}
\author{Thomas Powell}
\date{}
\begin{document}

\maketitle

%\author{Thomas Powell}

%\maketitle

\begin{abstract}
We study a convergence criterion which generalises the notion of being monotonically decreasing, and introduce a quantitative version of this criterion, a so called metastable rate of asymptotic decreasingness. We then present a concrete application in the fixed point theory of uniformly convex Banach spaces, in which we carry out a quantitative analysis of a convergence proof of Kirk and Sims. More precisely, we produce a rate of metastability (in the sense of Tao) for the Picard iterates of mappings which satisfy a variant of the convergence criterion, and whose fixed point set has nonempty interior.
\end{abstract}

\maketitle

%%%%%%%%%%%%%%%%%%%%%%%%%%%%%%%%%%%%%%%%%%%%%%%%%
%%%%%%%%%%%%%%%%%%%%%%%%%%%%%%%%%%%%%%%%%%%%%%%%%
\section{Introduction}
%%%%%%%%%%%%%%%%%%%%%%%%%%%%%%%%%%%%%%%%%%%%%%%%%
%%%%%%%%%%%%%%%%%%%%%%%%%%%%%%%%%%%%%%%%%%%%%%%%%
\label{sec-intro}

Let $T:C\to C$ be a mapping on some closed subset $C$ of a complete metric space $X$. It is well known that whenever $T$ is a contraction, the sequence $(T^nx)_{n\in\NN}$ of Picard iterates converges to a fixed point of $T$, and that this result no longer holds in general when $T$ is nonexpansive. A natural question then arises: under what additional conditions can we guarantee convergence of the Picard iterates for nonexpansive $T$?

For $X$ a Banach space, it turns out that a combination of uniform convexity of $X$ together with nonemptyness of the interior of the fixed point set $\fix{T}$ suffices for this - an observation originally due to Beauzamy. In \cite{KirSim(1999.0)}, Kirk and Sims provide a proof of this fact and observe that the result holds under much broader conditions on $T$, namely that $T$ (or one of its iterates) is continuous and
\begin{equation}
\label{eqn-Tconvinf}
\lim_{n\to\infty}\norm{p-T^nx}=\inf_{n\in\NN}\norm{p-T^nx}
\end{equation}
for all $p\in\fix{T}$. In particular, the above conditions hold whenever $T$ is asymptotically nonexpansive \cite{GoeKir(1972.0)}, in other words when there exists a sequence $(\mu_n)_{n\in\NN}$ with $\mu_n\to 1$ such that for all $x,y\in C$:
\begin{equation*}
\norm{T^nx-T^ny}\leq \mu_n\norm{x-y}.
\end{equation*}
In this paper, we study the result of Kirk and Sims in the context of the proof mining program. The novelty of the work lies in the fact that we carry out an abstract quantitative analysis of the notion of `convergence to the infimum', and that the subsequent application constitutes one of the first proof theoretic studies of a nonempty interior condition in the theory of uniformly convex Banach spaces.\\

\noindent Proof mining is a branch of \emph{applied proof theory} focused on the extraction of numerical information from nonconstructive proofs in mathematics. Developed in the 1990s by U. Kohlenbach, it has already achieved considerable success in functional analysis, where a corresponding body of so-called \emph{logical metatheorems} have now been established \cite{GerKoh(2008.0),Kohlenbach(2005.0)} which guarantee the extraction of highly uniform bounds from a large class of proofs. For a broader overview of the proof mining program, the reader is encouraged to consult the standard text \cite{Kohlenbach(2008.0)}.

The techniques of proof mining are often applied to proofs of convergence theorems, in which case we first have to specify what kind of numerical information we expect to be able to extract. Typically, this depends on the \emph{logical} structure of the theorem at hand. Convergence in general, when formulated as a Cauchy property, is a $\forall\exists \forall$ statement
\begin{equation}
\label{eqn-conv}
\forall \varepsilon>0\exists n\forall i,j\geq n(|x_i-x_j|\leq\varepsilon),
\end{equation}
and it is a well-known phenomenon that we cannot in general produce computable witnesses for theorems of this form. In the case of convergence this is witnessed by so-called Specker sequences \cite{Specker(1949.0)} - simple computable sequences $(x_n)_{n\in\NN}$ that do not have a computable rate of convergence i.e. a computable function $\phi(\varepsilon)$ satisfying
\begin{equation*}
\forall \varepsilon>0,i,j\geq \phi(\varepsilon)(|x_i-x_j|\leq\varepsilon).
\end{equation*}
That this is not possible in our specific setting, where $(x_n)_{n\in\NN}$ denotes a sequence of Picard iterates $(T^nx)_{n\in\NN}$, follows from \cite[Theorem 4.4 (2)]{Kohlenbach(2018.1)} which in turn adapts a construction of \cite{Neumann(2015.0)}: Already for $X=\RR$ it is shown that there exists a nonexpansive map $T:[0,1]\to [0,1]$ (which can be easily extended to one with $\inter{\fix{T}}\neq \emptyset$) such that the sequence $(T^n0)_{n\in\NN}$ has no computable rate of convergence.

In these situations, we consider instead the \emph{metastable} version of (\ref{eqn-conv}):
\begin{equation}
\label{eqn-metastablev}
\forall \varepsilon>0,\forall g:\NN\to\NN \;\exists  n \; \forall i,j\in [n,n+g(n)](|x_i-x_j|\leq \varepsilon).
\end{equation}
The statement (\ref{eqn-metastablev}) is known in logic as the Herbrand normal form of (\ref{eqn-conv}), and though the two are completely equivalent, (\ref{eqn-metastablev}) is computationally tractable in the sense that under very general conditions we can produce a computable rate of metastability for the $(x_i)$, that is a functional $\Phi(\varepsilon,g)$ satisfying
\begin{equation}
\label{eqn-metastable}
\forall \varepsilon>0,\forall g:\NN\to\NN\exists n\leq \Phi(\varepsilon,g)\forall i,j\in [n,n+g(n)](|x_i-x_j|\leq \varepsilon).
\end{equation}
Note that the existence of a computable rate of metastability does not contradict the impossibility of finding a computable rate of convergence, because the route from (\ref{eqn-metastablev}) to (\ref{eqn-conv}) requires proof by contradiction and is thus nonconstrutive. 

Recently, metastable convergence theorems in the sense of (\ref{eqn-metastablev}) have been made popular by T. Tao \cite{Tao(2008.1)} where they have found direct applications in ergodic theory. The extraction of effective rates of metastability is a standard result in proof mining, and has been accomplished in numerous different settings (see e.g. \cite{Kohlenbach(2005.1),KohKou(2015.0),KohLeu(2012.0)}).

The main result of this paper (Theorem \ref{res-mined}) is a highly uniform rate of metastability for the Picard iterates for mappings $T$ which satisfy (\ref{eqn-Tconvinf}), where $\inter{\fix{T}}\neq 0$ and $X$ is uniformly convex\footnote{Note that in the special case that $X$ is a Hilbert space and $T$ nonexpansive, convergence of the Picard iterates when $\inter{\fix{T}}\neq 0$ is an instance of a general convergence property for Fej\'er monotone sequences in Hilbert spaces \cite[Proposition 5.10]{BauCom(2010.0)}, and a rate of metastability for the Picard iterates in this case already follows from the analysis of this property given by Sipo\c s \cite[Proposition 4.2.3]{Sipos(2017.0)}.}. In order to achieve this, we need to provide \emph{quantitative analogues} for each of our assumptions. In particular,
\begin{itemize}

\item uniform convexity of $X$ is represented by a \emph{modulus of uniform convexity};

\item the property $\lim_{n\to\infty}\norm{p-T^nx}=\inf_{n\in\NN}\norm{p-T^nx}$ is represented by a \emph{metastable rate of asymptotic decreasingness}.

\end{itemize}

While moduli of uniform convexity are well known and have been widely applied in proof mining, metastable rates of asymptotic decreasingness are introduced and studied for the first time in this paper, and arise from a general analysis of convergence which we carry out in Section \ref{sec-inf}.

Rates of metastability are powerful in part because they enable us to produce \emph{direct} witnesses for $\forall \exists$ theorems. Generally speaking, if (\ref{eqn-conv}) is used as a lemma in the proof of some $\forall \exists$ theorem $\forall u\exists v P_0(u,v)$, a rate of metastability $\Phi$ can be used to construct a computable function $f_\Phi$ which satisfies $\forall u P_0(u,f_\Phi(u))$ (see \cite[Chapter 1]{Kohlenbach(2008.0)} for a more detailed discussion).

An illustration of this phenomenon can be given in our setting. For the special case where $T$ is nonexpansive, asymptotic regularity of the Picard iterates i.e.
\begin{equation*}
\forall \varepsilon>0\exists n\forall i\geq n(\norm{T^{i+1}x-T^ix}\leq\varepsilon)
\end{equation*}
follows directly from the fact that the iterates themselves converge. Since the sequence $(\norm{T^{i+1}x-T^ix})_{i\in\NN}$ is monotonically decreasing, asymptotic regularity is actually equivalent to the $\forall\exists$ statement
\begin{equation*}
\forall \varepsilon>0\exists n(\norm{T^{n+1}x-T^nx}\leq\varepsilon).
\end{equation*}
As a result, we are able to adapt our aforementioned rate of metastability to produce a direct rate of asymptotic regularity for the Picard iterates, and in Theorem \ref{res-lp} we provide a concrete instance of this rate in the case where $X$ is an $L_p$ space. This result forms yet another illustration of how methods from proof theory can be used to extract simple numerical bounds of low polynomial complexity from proofs which are prima-facie nonconstructive.

%%%%%%%%%%%%%%%%%%%%%%%%%%%%%%%%%%%%%%%%%%%%%%%%%
%%%%%%%%%%%%%%%%%%%%%%%%%%%%%%%%%%%%%%%%%%%%%%%%%
\section{Convergence to the infimum}
%%%%%%%%%%%%%%%%%%%%%%%%%%%%%%%%%%%%%%%%%%%%%%%%%
%%%%%%%%%%%%%%%%%%%%%%%%%%%%%%%%%%%%%%%%%%%%%%%%%
\label{sec-inf}

We begin by studying the notion that a sequence $(x_n)_{n\in\NN}$ of nonnegative real numbers converges to their infimum:
\begin{equation*}
\lim_{n\to\infty} x_n=\inf_{n\in\NN} x_n
\end{equation*}
This holds whenever $(x_n)_{n\in\NN}$ is decreasing, but conversely not all sequences converging to their infimum are decreasing. It turns out, however, that the following Cauchy property, which in some sense approximates decreasingness, is necessary and sufficient for convergence to the infimum:
\begin{equation}
\label{eqn-asdec}
\forall \varepsilon>0, N\exists n\forall i\geq n(x_i\leq x_N+\varepsilon).
\end{equation}
We refer to any sequence satisfying (\ref{eqn-asdec}) as \emph{\asym}. Note that in the case where $(x_n)_{n\in\NN}$ is just decreasing we would set $n=N$, although in general $n$ can also depend on $\varepsilon$. The following fact is trivial, but we isolate it since it will be separately analysed below.

\begin{lemma}
\label{lem-inf}
Let $(x_n)$ be a nonnegative sequence of real numbers. Then it satisfies 
\begin{equation*}
\forall \varepsilon>0\exists N\forall i(x_N-\varepsilon\leq x_i).
\end{equation*}
\end{lemma}

\begin{proof}

Let $d:=\inf_{n\in\NN} x_n$. Then for any $\varepsilon>0$ there exists some $N$ with $x_N-\varepsilon\leq d$ and thus $x_N-\varepsilon\leq x_i$ for all $i$.
\end{proof}

\begin{theorem}
\label{res-conv-asym}
A sequence $(x_n)_{n\in\NN}$ of nonnegative reals is asymptotically decreasing if and only if it converges to its infimum.
\end{theorem}

\begin{proof}
Let $d:=\inf_{n\in\NN} x_n$. If $\lim_{n\to\infty} x_n=d$ then for any $\varepsilon>0$ there is some $n$ such that $\forall i\geq n(x_i\leq d+\varepsilon)$, and since for any $N$ we have $d\leq x_N$ and hence $x_i\leq x_N+\varepsilon$, it follows easily that $(x_n)$ is asymptotically decreasing.

In the other direction, if $(x_n)$ is asymptotically decreasing then we first claim that it is Cauchy. To see this, fixing $\varepsilon>0$ we apply Lemma \ref{lem-inf} by which there is some $N$ such that $x_N-\varepsilon\leq x_i$ for all $i$. In addition, by (\ref{eqn-asdec}) there is some $n$ such that $x_i\leq x_N+\varepsilon$ for all $i\geq n$, and combining these we have $|x_i-x_N|\leq \varepsilon$ for all $i\geq n$ and thus $|x_i-x_j|\leq 2\varepsilon$ for all $i,j\geq n$.

This proves the claim, and thus $(x_i)$ converges to some limit $l\geq d$. Suppose that $l>d$ and let $\varepsilon:=(l-d)/4$. Then there is some $N$ such that $x_N\leq d+\varepsilon$, and by (\ref{eqn-asdec}) some $n$ such that $x_i\leq x_N+\varepsilon\leq d+2\varepsilon=l-2\varepsilon$ for all $i\geq n$, contradicting $x_n\to l$. Therefore $l=d$.
\end{proof}
In general, there are sequences of nonnegative reals which have no computable metastable rate of asymptotic decreasingness, by which we mean a procedure $\phi(\varepsilon,N)$ which returns some $n$ satisfying (\ref{eqn-asdec}). A trivial example would be to take any decreasing Specker sequence $(a_n)_{n\in\NN}$ which converges to some non-computable infimum $d$, and define $(x_n)_{n\in\NN}$ by $x_0:=d$ and $x_{n+1}:=a_n$. Then $\phi(\varepsilon,0)$ would be a computable rate of convergence for the Specker sequence, which is impossible by definition\footnote{A. Sipo\c s asks if we could strengthen this by choosing $(x_i)_{i\in\NN}$ be to \emph{computable}, since in our example, $x_0=d$ is a noncomputable real. We conjecture that this is indeed possible, but a more detailed discussion at this point would take us too far afield.}.

Thus we are interested in the metastable version of asymptotic decreasingness, namely the reformulation of (\ref{eqn-asdec}) analogous to metastability in the usual sense as given in (\ref{eqn-metastablev}), which in this case is
\begin{equation*}
\label{eqn-asymmetaz}
\forall \varepsilon>0,N,g:\NN\to\NN\;\exists n \;\forall i\in [n,n+g(n)](x_i\leq x_N+\varepsilon).
\end{equation*}
Analogous to (\ref{eqn-metastable}), we call $\Gamma(\varepsilon,g,N)$ a metastable rate of asymptotic decreasingness if it satisfies
\begin{equation}
\label{eqn-asymmeta}
\forall \varepsilon>0,N,g:\NN\to\NN\exists n\leq \Gamma(\varepsilon,g,N)\forall i\in [n,n+g(n)](x_i\leq x_N+\varepsilon).
\end{equation}
We now carry out a quantitative analysis of Theorem \ref{res-conv-asym} and show how to obtain a rate of metastability $\Phi(\varepsilon,g)$ for $(x_n)_{n\in\NN}$ in terms of some $\Gamma(\varepsilon,g,N)$ satisfying (\ref{eqn-asymmeta}). Our first step is a quantitative version of Lemma \ref{lem-inf}, for which we need some notation which we will also make use of in later sections.
\begin{notation}
\label{not-star}
\begin{itemize}

\item Given a function $f:\NN\to\NN$ we define $f^\ast:\NN\to\NN$ by $f^\ast(n):=\max_{i\leq n} \{n,f(i)\}$. Note that $n\leq f^\ast(n)$ and $f(n)\leq f^\ast(n)$ for all $n\in\NN$, and moreover $f^\ast(m)\leq f^\ast(n)$ whenever $m\leq n$. 

\item Similarly, given a functional $\Gamma(\varepsilon,g,N)$ we define $$\Gamma^\ast(\varepsilon,g,N):=\max_{i\leq N}\{N, \Gamma(\varepsilon,g,i)\}$$ and observe that it satisfies analogous monotonicity properties.

\item We denote by $f^{(n)}$ the $n$-times iteration of $f$.

\end{itemize}
\end{notation}

\begin{lemma}
\label{lem-infq}
Let $(x_n)_{n\in\NN}$ be a sequence of nonnegative real numbers with $x_0<K$. Then 
\begin{equation}
\label{eqn-infq}
\forall \varepsilon>0,f:\NN\to\NN \exists N\leq (f^\ast)^{(\ceil{K/\varepsilon})}(0)\forall i\leq f(N)(x_N-\varepsilon\leq x_i).
\end{equation}
\end{lemma}

\begin{proof}
Suppose that (\ref{eqn-infq}) is false for some $\varepsilon>0$ and $f$ i.e.
\begin{equation}
\label{eqn-neginfq}
\forall N\leq (f^\ast)^{(\ceil{K/\varepsilon})}(0)\exists i\leq f(N)(x_i<x_N-\varepsilon) 
\end{equation}
We show that for all $1\leq n\leq\ceil{K/\varepsilon}$ there exists some $j\leq (f^\ast)^{(n)}(0)$ such that $x_j<x_0-n\cdot\varepsilon$. For $n=1$ this follows by setting $N=0$ in (\ref{eqn-neginfq}) and noting that $x_j<x_0-\varepsilon$ for some $j\leq f(0)=f^\ast(0)$. For the induction step, assuming that for $n<\ceil{K/\varepsilon}$ we have found some $j\leq (f^\ast)^{(n)}(0)$ with $x_j<x_0-n\cdot\varepsilon$, we set $N:=j$. Since $(f^\ast)^{(n)}(0)\leq (f^\ast)^{(\ceil{K/\varepsilon})}(0)$ we can apply (\ref{eqn-neginfq}), by which there exists some $i\leq f(j)\leq f^\ast(j)\leq (f^\ast)^{(n+1)}(0)$ with $x_i<x_j-\varepsilon<x_0-(n+1)\cdot\varepsilon$.

This completes the induction, so for $n=\ceil{K/\varepsilon}$ there exists some $j$ such that $x_j<x_0-\ceil{K/\varepsilon}\cdot\varepsilon\leq 0$, a contradiction.
\end{proof}

\begin{theorem}
\label{res-cauchymined}
Let $(x_n)_{n\in\NN}$ be a sequence of nonnegative numbers with $x_0<K$, and suppose that $\Gamma(\varepsilon,g,N)$ satisfies (\ref{eqn-asymmeta}). Then the function
\begin{equation*}
\Phi(\varepsilon,g):=\Gamma^\ast(\tfrac{\varepsilon}{2},g,f^{(\ceil{2K/\varepsilon})}(0))
\end{equation*}
for $f(j):=\Gamma^\ast(\tfrac{\varepsilon}{2},g,j)+g^\ast(\Gamma^\ast(\tfrac{\varepsilon}{2},g,j))$ is a rate of metastability for $(x_n)_{n\in\NN}$ in the sense of (\ref{eqn-metastable}).
\end{theorem}

\begin{proof}
Fix some parameters $\varepsilon>0$ and $g$. It is easy to show that $n\leq f(n)$ and $m\leq n$ implies $f(m)\leq f(n)$, and thus $f^\ast(n)=f(n)$. Therefore by Lemma \ref{lem-infq} there exists some $N\leq f^{(\ceil{2K/\varepsilon})}(0)$ such that
\begin{equation*}
\forall i\leq f(N)(x_N-\tfrac{\varepsilon}{2}\leq x_i).
\end{equation*}
But by definition of $\Gamma$ there is some
\begin{equation*}
n\leq \Gamma(\tfrac{\varepsilon}{2},g,N)\leq \Gamma^\ast(\tfrac{\varepsilon}{2},g,N)\leq \Gamma^\ast(\tfrac{\varepsilon}{2},g,f^{(\ceil{2K/\varepsilon})}(0))=\Phi(\varepsilon,g)
\end{equation*}
satisfying $\forall i\in [n,n+g(n)](x_i\leq x_N+\tfrac{\varepsilon}{2})$, and since
\begin{equation*}
n+g(n)\leq n+g^\ast(n)\leq \Gamma^\ast(\tfrac{\varepsilon}{2},g,N)+g^\ast(\Gamma^\ast(\tfrac{\varepsilon}{2},g,N))=f(N)
\end{equation*}
we have in addition $x_N-\tfrac{\varepsilon}{2}\leq x_i$ and thus $|x_i-x_N|\leq\tfrac{\varepsilon}{2}$ for all $i\in [n,n+g(n)]$, from which the result follows.
\end{proof}

In the special case where $(x_n)_{n\in\NN}$ is decreasing, (\ref{eqn-asymmeta}) would be realised by the function $\Gamma(\varepsilon,g,N):=N$ and the rate of metastability obtained from Theorem \ref{res-cauchymined} would be $\Phi(\varepsilon,g)=f^{(\ceil{2K/\varepsilon})}(0)$ for $f(j):=j+g^\ast(j)$, which is (essentially) the standard rate of metastability for monotone sequences.  

%%%%%%%%%%%%%%%%%%%%%%%%%%%%%%%%%%%%%%%%%%%%%%%%%
%%%%%%%%%%%%%%%%%%%%%%%%%%%%%%%%%%%%%%%%%%%%%%%%%
\section{Asymptotically nonexpansive mappings in uniformly convex spaces}
%%%%%%%%%%%%%%%%%%%%%%%%%%%%%%%%%%%%%%%%%%%%%%%%%
%%%%%%%%%%%%%%%%%%%%%%%%%%%%%%%%%%%%%%%%%%%%%%%%%
\label{sec-state}

We now move on to an application of our metastable formulation of asymptotically decreasing sequences in the fixed point theory of nonexpansive mappings. More specifically, we will analyse the following result, which is adapted from Kirk and Sims \cite{KirSim(1999.0)}:
\begin{theorem}
\label{res-main}
Let $C$ be a subset of a uniformly convex Banach space $X$ and $T:C\to C$ a mapping with $\inter{\fix{T}}\neq\emptyset$. Pick some $x\in C$ and suppose that $T$ satisfies the following condition:
\begin{equation}
\label{eqn-asdecT}
\forall q\in\fix{T},\varepsilon>0,N\exists n\forall i\geq n(\norm{T^ix-q}\leq \norm{T^Nx-q}+\varepsilon).
\end{equation} 
Then the sequence $(T^nx)_{n\in\NN}$ is Cauchy.
\end{theorem}
The condition (\ref{eqn-asdecT}) simply says that the sequence $(\norm{T^nx-q})_{n\in\NN}$ is asymptotically decreasing for any $q\in\fix{T}$. It is trivially satisfied when $T$ is quasi-nonexpansive, and also more generally when $T$ is asymptotically nonexpansive:
\begin{lemma}
\label{lem-anone-asdecT}
Suppose that $T:C\to C$ is asymptotically nonexpansive w.r.t. some sequence $\mu_n\to 1$. Then $T$ satisfies (\ref{eqn-asdecT}) for any $x\in C$.
\end{lemma}

\begin{proof}
Fixing parameters $x\in C$, $q\in\fix{T}$, $\varepsilon>0$ and $N\in\NN$ and assuming w.l.o.g. that $\norm{T^Nx-q}>0$, there exists some $j$ such that 
\begin{equation*}
\mu_k\leq 1+\frac{\varepsilon}{\norm{T^Nx-q}}
\end{equation*}
for all $k\geq j$. Then for $i\geq N+j$ we have $i=N+k$ for some $k\geq j$ and thus
\begin{equation*}
\norm{T^ix-q}=\norm{T^{N+k}x-T^kq}\leq \mu_k\norm{T^Nx-q}\leq \norm{T^Nx-q}+\varepsilon.
\end{equation*}
\end{proof}
As an immediate consequence we obtain the following, which is given as Theorem 5.2 of \cite{KirSim(1999.0)}.
\begin{corollary}
[Kirk/Sims \cite{KirSim(1999.0)}]
\label{res-maincor}
Suppose that $C$ is a closed subset of a uniformly convex Banach space and $T:C\to C$ is asymptotically nonexpansive with $\inter{\fix{T}}\neq\emptyset$. Then for each $x\in C$ the sequence $(T^nx)_{n\in\NN}$ converges to a fixed point of $T$.
\end{corollary}

\begin{proof}
By Lemma \ref{lem-anone-asdecT}, $T$ satisfies (\ref{eqn-asdecT}) for any $x\in C$, and thus by Theorem \ref{res-main} $(T^nx)_{n\in\NN}$ is Cauchy, and thus converges to a limit which lies in the closed set $C$. Since $T$ is also continuous, this limit must be a fixed point.
\end{proof}
It is already observed in \cite{KirSim(1999.0)} that Corollary \ref{res-maincor} holds more generally for arbitrary $T$ satisfying
\begin{equation}
\label{eqn-liminf}
\lim_{n\to\infty}\norm{T^nx-q}=\inf_{n\in\NN}\norm{T^nx-q}
\end{equation}
for all $q\in\fix{T}$, provided $T$ or one of its iterates is continuous. In fact, only (\ref{eqn-liminf}) is required to establish convergence of the Picard iterates, and by Theorem \ref{res-conv-asym} this can be weakened to the Cauchy property (\ref{eqn-asdecT}). For all these results, uniform convexity of the underlying Banach space is used in the following form, due independently to Edelstein \cite{Edelstein(1966.0)} and Steckin \cite{Steckin(1963.0)}. 

\begin{notation}
We denote by $B_r[x]$ resp. $B^\circ_r[x]$  the closed resp. open ball of radius $r>0$ centered at $x$, and by $\seg{x,y}$ the line segment from $x$ to $y$.
\end{notation}

\begin{lemma}
\label{lem-balls}
Suppose that $X$ is a uniformly convex Banach space. Then for any $d>0$ and $c,c'\in X$ satisfying  $0<\norm{c-c'}=hd< d$ where $0<h<1$ we have
\begin{equation*}
\lim_{\delta\to 0}\diam\left(B_{d-hd+\delta}[c]\cap(X\backslash B_d^\circ[c'])\right)=0.
\end{equation*}
Moreover, the convergence is uniform in $c,c'$.
\end{lemma}

Note that later we analyse the proof of Lemma \ref{lem-balls} given in \cite{Vlasov(1973.0)} (cf. Lemma 1.1).

We now give a proof of Theorem \ref{res-main}, which is an adaptation of the proof of Theorem 5.2 in \cite{KirSim(1999.0)}. Crucially, we use separately the existence of $\inf_{n\in\NN}\norm{T^nx-p}$ for some $p\in\inter{\fix{T}}$ (which is of course trivial and has nothing to do with any properties of $T$), and later a single instance of (\ref{eqn-asdecT}) for some $q\neq p$.

\begin{proof}
[Proof of Theorem \ref{res-main}]
Since $\inter{\fix{T}}\neq\emptyset$ there is some $p\in\fix{T}$ and $r>0$ such that $B_r[p]\subset \fix{T}$. Let $d:=\inf_{n\in\NN}\norm{T^nx-p}$. If $d<r$ then the result follows trivially, so we assume $d\geq r$. 

Pick some $0<h<1$ such that $0<hd<r$ and thus $B_{hd}[p]\subset\fix{T}$. By Lemma \ref{lem-balls}, fixing any $\varepsilon>0$ there is some $\delta(\varepsilon,h,d)>0$ such that
\begin{equation}
\label{eqn-main0}
\diam\left(B_{d-hd+\delta}[c]\cap(X\backslash B_d^\circ[c'])\right)<\varepsilon
\end{equation}
for any $c,c'\in X$ with $\norm{c-c'}=hd$. Now, for each $n\in\NN$ choose $q_n\in \seg{p,T^nx}$ so that it satisfies
\begin{equation*}
\norm{q_n-p}=hd\mbox{ \ \ \ and therefore \ \ \ }\norm{T^nx-q_n}=\norm{T^nx-p}-hd.
\end{equation*}
Then we have $q_n\in\fix{T}$ for all $n\in\NN$ and $\inf_{n\in\NN}\norm{T^nx-q_n}=d-hd$, and so in particular there exists some $N$ such that
\begin{equation*}
\norm{T^Nx-q_N}\leq d-hd+\delta/2.
\end{equation*}
Applying (\ref{eqn-asdecT}) on $q_N\in\fix{T}$, $\delta/2$ and $N$ there exists some $n$ such that for all $i\geq n$ we have
\begin{equation*}
\norm{T^ix-q_N}\leq \norm{T^Nx-q_N}+\delta/2\leq d-hd+\delta.
\end{equation*}
In other words, $T^ix\in B_{d-hd+\delta}[q_N]$ for all $i\geq n$. But $\norm{T^ix-p}\geq d$ and thus we also have $T^ix\in X\backslash B^\circ_d[p]$ for all $i\geq n$, and since $\norm{q_N-p}=hd$ it follows by (\ref{eqn-main0}) that
\begin{equation*}
\norm{T^ix-T^jx}\leq \diam\left(B_{d-hd+\delta}[q_N]\cap(X\backslash B_d^\circ[p])\right)<\varepsilon
\end{equation*}
for all $i,j\geq n$. Since $\varepsilon>0$ was arbitrary, we have shown that the Picard iterates form a Cauchy sequence.
\end{proof}

We now carry out a quantitative analysis of the above proof. To begin with, we first need to consider the role played by uniform convexity in Lemma \ref{lem-balls}.

%%%%%%%%%%%%%%%%%%%%%%%%%%%%%%%%%%%%%%%%%%%%%%%%%
%%%%%%%%%%%%%%%%%%%%%%%%%%%%%%%%%%%%%%%%%%%%%%%%%
\section{A reformulation of uniform convexity}
%%%%%%%%%%%%%%%%%%%%%%%%%%%%%%%%%%%%%%%%%%%%%%%%%
%%%%%%%%%%%%%%%%%%%%%%%%%%%%%%%%%%%%%%%%%%%%%%%%%
\label{sec-conv}

Uniform convexity of a Banach space $X$ can be given a quantitative form via the so-called modulus of uniform convexity $\delta_X:(0,2]\to (0,1]$ defined by
\begin{equation*}
\delta_X(\varepsilon):=\inf\left\{1-\tfrac{1}{2}\norm{x_1+x_2}\; : \; x_1,x_2\in B, \norm{x_1-x_2}\geq \varepsilon\right\}
\end{equation*}
where $B:=B_1[0]$ is the closed unit ball. In applications of proof theory to uniformly convex spaces, rather than the unique modulus of uniform convexity $\delta_X$ one traditionally works with `a' modulus of uniform convexity for $X$, which is defined to be \emph{any} function $\Phi:(0,2]\to (0,1]$ satisfying
\begin{equation}
\label{eqn-unif}
\forall x_1,x_2\in B,\varepsilon\in (0,2]\left(\tfrac{1}{2}\norm{x_1+x_2}\geq 1-\Phi(\varepsilon)\to\norm{x_1-x_2}\leq\varepsilon\right).
\end{equation}
In our analysis of Theorem \ref{res-main}, we do not use such a modulus directly, but require instead a reformulation of it which reflects the way in which uniform convexity is applied via Lemma \ref{lem-balls}. The following is a quantitative analysis of the proof of Lemma \ref{lem-balls} given in \cite{Vlasov(1973.0)}.
\begin{lemma}
\label{lem-nestmod}
Suppose that $\Phi:(0,2]\to (0,1]$ is a modulus of uniform convexity for $X$, and define $\Psi: (0,\tfrac{1}{2})\times (0,4]\to (0,1]$ by
\begin{equation*}
\Psi(h,\varepsilon):=\min\{\tfrac{\varepsilon}{2},2h\Phi(\tfrac{\varepsilon}{2})\}.
\end{equation*}
Then $\Psi$ satisfies
\begin{equation}
\label{eqn-balls}
\begin{aligned}
&\forall y\in B,u\in X,h\in (0,\tfrac{1}{2}),\varepsilon\in (0,4]\\
&\left(\norm{u-hy}\leq 1-h+\Psi(h,\varepsilon)\wedge \norm{u}\geq 1\to \norm{u-y}\leq\varepsilon\right)
\end{aligned}
\end{equation} 
\end{lemma}

\begin{proof}
Fix parameters $y,u,h$ and $\varepsilon$ satisfying the premise of (\ref{eqn-balls}) and set $\delta:=\Psi(h,\varepsilon)$ and $v:=u/\norm{u}$. Note that $u\in B_{1-h+\delta}[hy]$ by definition, and since $h<\tfrac{1}{2}$ we have $h\leq 1-h+\delta$ and thus $0\in B_{1-h+\delta}[hy]$. Therefore in fact $\seg{0,u}\subset B_{1-h+\delta}[hy]$, and observing that $\norm{v}=1\leq\norm{u}$ we have $v\in\seg{0,u}$ and thus $v\in B_{1-h+\delta}[hy]$. It follows that
\begin{equation*}
\begin{aligned}
2h-\delta&=h+1-(1-h+\delta)\\
&\leq h+1-\norm{v-hy}\\
&=\norm{(h+1)v}-\norm{(h+1)v-h(y+v)}\\
&\leq h\norm{y+v}
\end{aligned}
\end{equation*} 
which, using that $\delta\leq  2h\Phi(\tfrac{\varepsilon}{2})$, results in
\begin{equation*}
\tfrac{1}{2}\norm{y+v}\geq 1-\frac{\delta}{2h}\geq 1-\Phi(\tfrac{\varepsilon}{2}).
\end{equation*}
Since $v,y\in B$ and $\tfrac{\varepsilon}{2}\in (0,2]$ we can apply (\ref{eqn-unif}) to obtain $\norm{y-v}\leq\tfrac{\varepsilon}{2}$. Observing that
\begin{equation*}
\norm{u-v}=\norm{u}-1\leq \norm{u-hy}+\norm{hy}-1\leq 1-h+\delta+h-1\leq\delta\leq\tfrac{\varepsilon}{2}
\end{equation*}
we have
\begin{equation*}
\norm{u-y}\leq \norm{u-v}+\norm{v-y}\leq \varepsilon
\end{equation*}
which completes the proof.
\end{proof}

%%%%%%%%%%%%%%%%%%%%%%%%%%%%%%%%%%%%%%%%%%%%%%%%%
%%%%%%%%%%%%%%%%%%%%%%%%%%%%%%%%%%%%%%%%%%%%%%%%%
\section{A quantitative analysis of Theorem \ref{res-main}}
%%%%%%%%%%%%%%%%%%%%%%%%%%%%%%%%%%%%%%%%%%%%%%%%%
%%%%%%%%%%%%%%%%%%%%%%%%%%%%%%%%%%%%%%%%%%%%%%%%%
\label{sec-main}

We now present the main result of the second part of this paper: A quantitative formulation of Theorem \ref{res-main}. Our goal is a rate of metastability for the Picard iterates $(T^nx)_{n\in\NN}$. However, we will need to take as input quantitative versions of the assumptions of Theorem \ref{res-main}.

For the remainder of the section, we fix our underlying Banach space $X$, together with $T:C\to C$ for some $C\subseteq X$ and $x\in C$. In addition, we assume we have all of the following, which form quantitative versions of our assumptions that (a) $X$ is uniformly convex, (b) $\inter{\fix{T}}\neq\emptyset$ and (c) $T$ satisfies (\ref{eqn-asdecT}) w.r.t $x$, respectively:
\begin{enumerate}[(a)]

\item A function $\Psi: (0,\tfrac{1}{2})\times (0,4]\to (0,1]$ as in Lemma \ref{lem-nestmod} satisfying (\ref{eqn-balls}), defined in terms of a modulus of uniform convexity $\Phi$ for $X$;

\item Some $p\in\fix{T}$ and $r>0$ such that $B_r[p]\subset \fix{T}$, together some $K>0$ satisfying $\norm{x-p}<K$;

\item A metastable rate of asymptotic decreasingness for $(\norm{T^nx-q})_{n\in\NN}$ which is \emph{uniform} in $q\in B_r[p]$ and depends only on $K$ and $r$ i.e. a function $\Gamma(K,r,\varepsilon,g,N)$ satisfying
\begin{equation}
\label{eqn-condc}
\begin{aligned}
&\forall\varepsilon>0,g:\NN\to\NN,N \\
&\exists n\leq \Gamma(K,r,\varepsilon,g,N)\forall i\in [n,n+g(n)](\norm{T^ix-q}\leq \norm{T^Nx-q}+\varepsilon)
\end{aligned}
\end{equation}
whenever $q\in B_r[p]$.

\end{enumerate}
As we will see in Section \ref{sec-special}, the additional uniformity conditions on $\Gamma$ - which simplify the analysis considerably - are naturally satisfied in the case of asymptotically nonexpansive mappings, and as such we consider them reasonable. Note that for $T$ nonexpansive, we can just set $\Gamma(K,r,\varepsilon,g,N):=N$, since for any $i=N+j\geq N$ and $q\in B_r[p]$ we trivially have
\begin{equation*}
\norm{T^{N+j}x-q}=\norm{T^{N+j}x-T^jq}\leq \norm{T^Nx-q}\leq \norm{T^Nx-q}+\varepsilon.
\end{equation*}
Our main result in Theorem \ref{res-mined} will be a rate of metastability $\Omega(\Phi,\Gamma,K,r,\varepsilon,g)$ for $(T^nx)_{n\in\NN}$, which in addition to the functions witnessing uniform convexity and metastable asymptotic decreasingness as above, depends only on sparse numerical data about $T,C,x$ and $\inter{\fix{T}}$, namely the radius of $B_r[p]\subset \fix{T}$ and an upper bound for $\norm{x-p}$.

For the remainder of the section, in addition to $X,T:C\to C$ and $x$, we fix $\Psi,\Gamma,p,K$ and $r$ satisfying (a)-(c), together with parameters $\varepsilon>0$ and $g:\NN\to\NN$. Our aim is to find a bound on some $n$ satisfying
\begin{equation*}
\forall i,j\in [n,n+g(n)](\norm{T^ix-T^jx}\leq \varepsilon),
\end{equation*}
which will be accomplished by a careful analysis of the proof of Theorem \ref{res-main} given in Section \ref{sec-state}. The first step is to eliminate the simple case in which $\norm{T^nx-p}<r$ for some $n\in\NN$.
\begin{lemma}
\label{lem-triv}
If $\norm{T^nx-p}<r$ for some $n\in\NN$ then we have $\forall i,j\in [n,n+g(n)](\norm{T^ix-T^jx}\leq\varepsilon)$.
\end{lemma}
\begin{proof}
$\norm{T^nx-p}<r$ implies that $T^nx\in\fix{T}$ and thus $T^ix=T^jx=T^nx$ for any $i,j\geq n$.
\end{proof}
Just as in the proof of Theorem \ref{res-main}, the main part of the analysis deals with the case $\norm{T^nx-p}\geq r$ for all $n\in\NN$. From now on, it will be helpful to use the following abbreviations:
\begin{equation}
\label{eqn-hdelta}
\begin{aligned}
h&:=\min\{\tfrac{1}{4},\tfrac{r}{K}\}\\
\delta&:=\min\{1,\Psi(h,\tfrac{\varepsilon}{2K})\}
\end{aligned}
\end{equation}
Note that $\delta$ is well defined since we can assume that $\varepsilon$ is small enough that $\tfrac{\varepsilon}{2K}\in (0,4]$. The following very simple lemma helps organise the way uniform convexity is applied.
\begin{lemma}
\label{lem-uc}
Suppose that $d\in \RR$ and $u,y\in X$ satisfy
\begin{enumerate}[(a)]

\item\label{item-uca} $0<d\leq K$

\item\label{item-ucb} $\norm{y}=d$

\item\label{item-ucc} $\norm{u}\geq d$

\item\label{item-ucd} $\norm{u-hy}\leq d(1-h+\delta)$.

\end{enumerate}
Then $\norm{u-y}\leq\tfrac{\varepsilon}{2}$.
\end{lemma}

\begin{proof}
Dividing the inequality (\ref{item-ucd}) through by $d>0$ we have
\begin{equation*}
\norm{u/d-hy/d}\leq 1-h+\delta\leq 1-h+\Psi(h,\tfrac{\varepsilon}{2K}).
\end{equation*}
Since, in addition, we have $\norm{u/d}\geq 1$ and $y/d\in B$, we can apply (\ref{eqn-balls}) to obtain
\begin{equation*}
\norm{u/d-y/d}\leq\frac{\varepsilon}{2K}\mbox{ \ \ and therefore \ \ }\norm{u-y}\leq \frac{\varepsilon}{2}
\end{equation*}
where for the last step we use $d\leq K$.
\end{proof}
We now deal with the main non-constructive step in the proof, namely the combination of use of $\inf_{n\in\NN}\norm{T^nx-p}$ in conjunction with asymptotic decreasingness of $(\norm{T^nx-q})_{n\in\NN}$ for all $q\in B_r[p]$. This step corresponds closely to the abstract result already established in Theorem \ref{res-cauchymined}.
\begin{lemma}
\label{lem-nc}
Take some $\eta>0$ and define the function $f:\NN\to\NN$ by
\begin{equation*}
f(j):=\Gamma^\ast(K,r,\eta,g,j)+g^\ast(\Gamma^\ast(K,r,\eta,g,j))
\end{equation*}
where $g^\ast$ and $\Gamma^\ast$ are defined as in Section 2 (where now $\Gamma$ has two additional fixed parameters $K,r$). Then there exists some
\begin{equation*}
N\leq f^{(\ceil{K/\eta})}(0)
\end{equation*}
together with some
\begin{equation*}
n\leq\Gamma^\ast(K,r,\eta,g,f^{(\ceil{K/\eta})}(0))
\end{equation*}
satisfying
\begin{enumerate}[(i)]

\item\label{item-nci} $\forall i\leq n+g(n)(\norm{T^Nx-p}-\eta\leq\norm{T^ix-p})$

\item\label{item-ncii} $\forall q\in B_r[p],i\in [n,n+g(n)](\norm{T^ix-q}\leq \norm{T^Nx-q}+\eta)$

\end{enumerate}
\end{lemma}

\begin{proof}
First of all, just as in the proof of Theorem \ref{res-cauchymined}, we note that $f^\ast(n)=f(n)$. Therefore applying Lemma \ref{lem-infq} to the sequence $x_n:=\norm{T^nx-p}$ and noting that $x_0=\norm{x-p}<K$, there exists some $N\leq f^{(\ceil{K/\eta})}(0)$ such that
\begin{equation}
\label{eqn-nc0}
\forall i\leq f(N)(\norm{T^Nx-p}-\eta\leq \norm{T^ix-p}).
\end{equation}
Now, by the defining property of $\Gamma$ (\ref{eqn-condc}), there exists some
\begin{equation*}
n\leq \Gamma(K,r,\eta,g,N)\leq \Gamma^\ast(K,r,\eta,g,N)\leq \Gamma^\ast(K,r,\eta,g,f^{(\ceil{K/\eta})}(0))
\end{equation*}
satisfying
\begin{equation*}
\forall i\in [n,n+g(n)](\norm{T^ix-q}\leq \norm{T^Nx-q}+\eta)
\end{equation*}
for all $q\in B_r[p]$. This establishes (\ref{item-ncii}). But
\begin{equation*}
n+g(n)\leq n+g^\ast(n)\leq \Gamma^\ast(K,r,\eta,g,N)+g^\ast(\Gamma^\ast(K,r,\eta,g,N))=f(N)
\end{equation*}
and thus (\ref{item-nci}) follows from (\ref{eqn-nc0}).
\end{proof}
We now tie together the previous two lemmas.
\begin{lemma}
\label{lem-tie}
Suppose that $N$ and $n$ are as in Lemma \ref{lem-nc} for
\begin{equation*}
\eta:=\frac{r\delta}{4}
\end{equation*}
and that
\begin{equation}
\label{eqn-inf}
\norm{T^Nx-p}\geq r.
\end{equation}
Define $d\in\RR$ together with $y\in X$ as follows:
\begin{equation*}
\begin{aligned}
d&:=\norm{T^Nx-p}-\eta\\
y&:=d\left(\frac{T^Nx-p}{\norm{T^Nx-p}}\right).
\end{aligned}
\end{equation*}
Then for any $i\in [n,n+g(n)]$ we have
\begin{equation*}
\norm{T^ix-(p+y)}\leq\frac{\varepsilon}{2}.
\end{equation*}
\end{lemma}

\begin{proof}
First, from our assumption (\ref{eqn-inf}) we have
\begin{equation}
\label{eqn-tie0}
d=\norm{T^Nx-p}-\eta\geq r-\frac{r\delta}{4}>\frac{r}{2}
\end{equation}
where the last step follows from $\delta\leq 1$. This will be used in two places in the proof. We now fix some $i\in [n,n+g(n)]$ and define $u_i:=T^ix-p$. We will show that $d,u_i$ and $y$ satisfy each of the conditions from Lemma \ref{lem-uc}.

For (\ref{item-uca}), $d>0$ follows from (\ref{eqn-tie0}), and using Lemma \ref{lem-nc} (\ref{item-nci}) we have in particular
\begin{equation*}
d=\norm{T^Nx-p}-\eta\leq \norm{T^0x-p}=\norm{x-p}<K.
\end{equation*}
Condition (\ref{item-ucb}) holds by definition, while for (\ref{item-ucc}) we use Lemma \ref{lem-nc} (\ref{item-nci}) again: Since $i\leq n+g(n)$ we have
\begin{equation*}
\norm{u_i}=\norm{T^ix-p}\geq \norm{T^Nx-p}-\eta=d.
\end{equation*}
The final condition (\ref{item-ucd}) is a little more involved. First let us define $q:=p+hy$. Since
\begin{equation*}
\norm{p-q}=\norm{hy}=hd\leq \frac{rd}{K}\leq r
\end{equation*}
here using that $h\leq\frac{r}{K}$ and $d\leq K$, we thus have $q\in B_r[p]$, and so now appealing to Lemma \ref{lem-nc} (\ref{item-ncii}) we have
\begin{equation*}
\begin{aligned}
\norm{u_i-hy}&=\norm{T^ix-p-hy}\\
&=\norm{T^ix-q}\\
&\stackrel{(\ref{item-ncii})}{\leq}\norm{T^Nx-q}+\eta\\
&=\norm{T^Nx-p-hy}+\eta\\
&=\norm{(T^Nx-p)-hd\left(\frac{T^Nx-p}{\norm{T^Nx-p}}\right)}+\eta\\
&=\left(1-\frac{hd}{\norm{T^Nx-p}}\right)\norm{T^Nx-p}+\eta\\
&=\norm{T^Nx-p}-hd+\eta\\
&=d-hd+2\eta\\
&=d-hd+\frac{r\delta}{2}\\
&<d(1-h+\delta)
\end{aligned}
\end{equation*}
where the last step uses (\ref{eqn-tie0}). Since the hypotheses of Lemma \ref{lem-uc} are now all satisfied, we can infer 
\begin{equation*}
\norm{T^ix-(p+y)}=\norm{u_i-y}\leq \frac{\varepsilon}{2}.
\end{equation*}
and we are done.
\end{proof}

\begin{corollary}
\label{res-tiecauchy}
Suppose that $n$ and $N$ are defined as in Lemma \ref{lem-nc} for $\eta:=\frac{r}{4}\cdot\min\{1,\Psi(\min\{\tfrac{1}{4},\frac{r}{K}\},\frac{\varepsilon}{2k})\}$ and that $\norm{T^Nx-p}\geq r$. Then we have $\forall i,j\in [n,n+g(n)](\norm{T^ix-T^jx}\leq\varepsilon)$.
\end{corollary}

\begin{proof}
Note that $\eta=\frac{r\delta}{4}$ for $\delta$ as defined in (\ref{eqn-hdelta}). Defining $d$ and $y$ as in Lemma \ref{lem-tie} it follows that for $i,j\in [n,n+g(n)]$ we have $
\norm{T^ix-T^jx}\leq \norm{T^ix-(p+y)}+\norm{(p+y)-T^jx}\leq \varepsilon$.
\end{proof}
The main theorem now follows by combining the above corollary with Lemma \ref{lem-triv} and expanding all of our definitions.

\begin{theorem}
\label{res-mined}
Let $X$ be a Banach space, $T:C\to C$ a mapping such that $B_r[p]\subset\fix{T}$ for some $p\in X$ and $r>0$. Suppose that $x\in C$ and $\norm{x-p}<K$ for some $K$. Finally, suppose that $\Phi$ is a modulus of uniform convexity for $X$ in the sense of (\ref{eqn-unif}) and $\Gamma$ a rate of asymptotic decreasingness for $(\norm{T^nx-q})_{n\in\NN}$ uniform for $q\in B_r[p]$ as in (\ref{eqn-condc}). Then
\begin{equation*}
\forall \varepsilon>0,g:\NN\to\NN\exists n\leq\Omega(\Phi,\Gamma,K,r,\varepsilon,g)\forall i,j\in [n,n+g(n)](\norm{T^ix-T^jx}\leq\varepsilon)
\end{equation*}
where $\Omega$ is defined as follows:
\begin{itemize}

\item $\Omega(\Phi,\Gamma,K,r,\varepsilon,g):=\Gamma^\ast(K,r,\eta,g,f^{(\ceil{K/\eta})}(0))$;

\item $f(j):=\Gamma^\ast(K,r,\eta,g,j)+g^\ast(\Gamma^\ast(K,r,\eta,g,j))$;

\item $\eta:=\frac{r}{4}\cdot\min\{1,\Psi(\min\{\tfrac{1}{4},\frac{r}{K}\},\frac{\varepsilon}{2K})\}$;

\item $\Psi(h,\varepsilon'):=\min\{\frac{\varepsilon'}{2},2h\Phi(\frac{\varepsilon'}{2})\}$.

\end{itemize}
\end{theorem}

\begin{proof}
There are two possibilities: If $\norm{T^nx-p}< r$ for some $n\leq f^{(\ceil{K/\eta})}(0)$ then by Lemma \ref{lem-triv} we would have $\norm{T^ix-T^jx}\leq\varepsilon$ for all $i,j\in [n,n+g(n)]$, and since $f^{(\ceil{K/\eta})}(0)\leq\Omega(\Phi,\Gamma,K,r,\varepsilon,g)$ the result follows.

On the other hand, suppose that this is not the case, and $n$ and $N$ are as in Lemma \ref{lem-nc}. In particular, since $N\leq f^{(\ceil{K/\eta})}(0)$ we must have $\norm{T^Nx-p}\geq r$, and so $\norm{T^ix-T^jx}\leq\varepsilon$ for all $i,j\in [n,n+g(n)]$ by Corollary \ref{res-tiecauchy}, and $n\leq \Omega(\Phi,\Gamma,K,r,\varepsilon,g)$ by Lemma \ref{lem-nc}.
\end{proof}

Note that as a simple corollary to our main theorem, we are able to locate \emph{approximate} fixed points $x_\varepsilon$ of $T$, where $x_\varepsilon$ is an approximate, or $\varepsilon$-fixed point if
\begin{equation*}
\norm{Tx_\varepsilon-x_\varepsilon}\leq\varepsilon.
\end{equation*}
\begin{corollary}
\label{res-appfix}
Under the conditions of Theorem \ref{res-mined}, for any $x\in C$ and $\varepsilon>0$, there exists some $n\leq \Omega(\Phi,\Gamma,K,r,\varepsilon,g)$ where $g(k):=k+1$ such that $x_n=T^nx$ is an $\varepsilon$-fixed point of $T$.
\end{corollary}

\begin{proof}
By Theorem \ref{res-mined}, there exists some $n\leq \Omega(\Phi,\Gamma,K,r,\varepsilon,g)$ such that for all $i,j\in [n,n+1](\norm{T^ix-T^jx}\leq\varepsilon)$, so in particular for $i=n+1,j=n$ we have $\norm{T^nx_n-x_n}=\norm{T^{(n+1)}x-T^nx}\leq\varepsilon$.
\end{proof}

%%%%%%%%%%%%%%%%%%%%%%%%%%%%%%%%%%%%%%%%%%%%%%%%%
%%%%%%%%%%%%%%%%%%%%%%%%%%%%%%%%%%%%%%%%%%%%%%%%%
\section{Special cases}
%%%%%%%%%%%%%%%%%%%%%%%%%%%%%%%%%%%%%%%%%%%%%%%%%
%%%%%%%%%%%%%%%%%%%%%%%%%%%%%%%%%%%%%%%%%%%%%%%%%
\label{sec-special}

We conclude by considering concrete properties on the mapping $T:C\to C$ which guarantee that there is some metastable rate of asymptotic decreasingness $\Gamma$ satisfying (\ref{eqn-condc}).

%%%%%%%%%%%%%%%%%%%%%%%%%%%%%%%%%%%%%%%%%%%%%%%%%
\subsection{Nonexpansive mappings}
%%%%%%%%%%%%%%%%%%%%%%%%%%%%%%%%%%%%%%%%%%%%%%%%%
\label{sec-special-nonex}

As we have already mentioned, in the case where $T$ is nonexpansive, the function $\Gamma(K,r,\varepsilon,g,N):=N$ forms a metastable rate of asymptotic decreasingness. In this case, a rate of metastability for the Picard iterates is given by
\begin{equation}
\label{eqn-nonex}
\Omega(\Phi,K,r,\varepsilon,g):=\tilde{g}^{(\ceil{K/\eta})}(0)
\end{equation}
for $\tilde g(j):=j+g^\ast(j)$ and $\eta$ defined as in Theorem \ref{res-mined}. In fact, one can even show that the same bound works with $\tilde g(j):=j+g(j)$, but we do not give the details here. Moreover, in the case of nonexpansive maps, Corollary \ref{res-appfix} even gives us a direct rate of asymptotic regularity as a byproduct. Recall that a mapping $T$ is asymptotically regular if
\begin{equation*}
\forall \varepsilon>0\exists n\forall i\geq n(\norm{T^{i+1}x-T^ix}\leq\varepsilon).
\end{equation*}
Now by Corollary \ref{res-appfix}, setting $g(k):=k+1$ in (\ref{eqn-nonex}) and noting that in this case $\tilde g^{(m)}(k)=k+m$, it follows that there exists some $n\leq \ceil{K/\eta}$ such that $\norm{T^{n+1}x-T^nx}\leq\varepsilon$. But since $T$ is nonexpansive, for any $i\geq n$ we have $\norm{T^{i+1}x-T^ix}\leq \norm{T^{n+1}x-T^nx}\leq\varepsilon$, and so $f(\varepsilon):=\ceil{K/\eta}$ is a rate of asymptotic regularity for $T$. Simple, concrete instantiations of this can be given in particular spaces, for example:
\begin{theorem}
\label{res-lp}
Let $T:C\to C$ be a nonexpansive mapping in $L_p$ for $2\leq p<\infty$, and suppose that $B_r[q]\subset \fix{T}$ for some $q\in L_p$ and $r>0$. Suppose that $x\in C$ and $\norm{x-q}<K$. Then 
\begin{equation*}
\forall \varepsilon>0, i\geq f(\varepsilon)(\norm{T^{i+1}x-T^ix}\leq\varepsilon)
\end{equation*}
where
\begin{equation*}
f(\varepsilon):=\left\lceil\frac{p\cdot 2^{3p+1}\cdot K^{p+2}}{\varepsilon^p\cdot r^2}\right\rceil
\end{equation*}
\end{theorem}

\begin{proof}
For $2\leq p<\infty$ a modulus of uniform convexity of $L_p$ is given by $\Phi(\varepsilon)=\frac{\varepsilon^p}{p2^p}$ (see \cite[p. 63]{LinTza(1979.0)}). The result follows by the observations above and unwinding the definitions of Theorem \ref{res-mined}.
\end{proof}

%%%%%%%%%%%%%%%%%%%%%%%%%%%%%%%%%%%%%%%%%%%%%%%%%
\subsection{Asymptotically nonexpansive mappings}
%%%%%%%%%%%%%%%%%%%%%%%%%%%%%%%%%%%%%%%%%%%%%%%%%
\label{sec-special-asymp}

The case where $T$ is asymptotically nonexpansive w.r.t. some $\mu_n\to 1$ is slightly more interesting, as in general $\Gamma$ will now depend on information about the convergence of $(\mu_n)_{n\in\NN}$. The following is essentially a quantitative version of Lemma \ref{lem-anone-asdecT}.

\begin{lemma}
\label{lem-anonemined}
Suppose that $T:C\to C$ is asymptotically nonexpansive w.r.t. $\mu_n\to 1$ with $B_r[p]\subset \fix{T}$ and $x\in C$ is such that $\norm{x-p}<K$. Suppose in addition that $(\mu_n)_{n\in\NN}$ is bounded above by some $L>0$ and has a rate of metastability $\phi$ which satisfies
\begin{equation*}
\forall\delta>0,h:\NN\to\NN\exists k\leq\phi(\delta,h)\forall m\in [k,k+h(k)](\mu_m\leq 1+\delta).
\end{equation*}
Then the function $\Gamma$ given by
\begin{equation*}
\Gamma_{L,\phi}(K,r,\varepsilon,g,N):=N+\phi\left(\frac{\varepsilon}{L(K+r)},g_N\right)
\end{equation*}
where $g_N(k):=g(N+k)$, is a metastable rate of asymptotic decreasingness satisfying (\ref{eqn-condc}), uniformly for $q\in B_r[p]$.
\end{lemma}

\begin{proof}
Fixing some $\varepsilon>0,g:\NN\to\NN,N$ and $q\in B_r[p]\subset\fix{T}$, we first note that there exists some $k\leq \phi\left(\frac{\varepsilon}{L(K+r)},g_N\right)$ such that
\begin{equation}
\label{eqn-anonemined0}
\forall m\in [k,k+g(N+k)]\left(\mu_m\leq 1+\frac{\varepsilon}{L(K+r)}\right).
\end{equation}
Setting $n:=N+k\leq \Gamma_{L,\phi}(K,r,\varepsilon,g,N)$, for any $i\in [n,n+g(n)]$ we have $i=N+k+j$ for some $0\leq j\leq g(N+k)$, and therefore
\begin{equation*}
\begin{aligned}
\norm{T^ix-q}&=\norm{T^{N+k+j}x-T^{k+j}q}\\
&\leq \mu_{k+j}\norm{T^Nx-q}\\
&\leq \norm{T^Nx-q}+\frac{\varepsilon \norm{T^Nx-q}}{L(K+r)}
\end{aligned}
\end{equation*}
where
\begin{equation*}
\begin{aligned}
\frac{\varepsilon \norm{T^Nx-q}}{L(K+r)}&\leq \frac{\varepsilon \mu_N\norm{x-q}}{L(K+r)}\\
&\leq \frac{\varepsilon \mu_N(\norm{x-p}+\norm{p-q})}{L(K+r)}\\
&\leq \frac{\varepsilon L(K+r)}{L(K+r)}=\varepsilon
\end{aligned}
\end{equation*}
and thus we've shown $\norm{T^ix-q}\leq \norm{T^Nx-q}+\varepsilon$.
\end{proof}

Therefore instantiating Theorem \ref{res-mined} with $\Gamma$ as defined in Lemma \ref{lem-anonemined} would result in a rate of metastability $\Omega_{L,\phi}(\Phi,K,r,\varepsilon,g)$ for the Picard iterates in the case where $T$ is asymptotically nonexpansive w.r.t. $\mu_n\to 1$, in terms of the usual data plus an upper bound $L$ and rate of metastability $\phi$ for $(\mu_n)_{n\in\NN}$.

Note that it is far more common for asymptotically nonexpansive maps to be defined w.r.t. some decreasing $(\mu_n)_{n\in\NN}$, in which case convergence $\mu_n\to 1$ is actually a $\forall\exists$ formula 
\begin{equation*}
\forall\delta>0\exists k(\mu_k\leq 1+\delta)
\end{equation*}
and from a proof of the above we would generally expect to be able to produce an explicit rate of convergence $c(\delta)$ for $(\mu_n)_{n\in\NN}$ satisfying
\begin{equation*}
\forall \delta>0(\mu_{c(\delta)}\leq 1+\delta).
\end{equation*}
In this case, a simplification of the proof of Lemma \ref{lem-anonemined} would result in a rate of asymptotic decreasingness given by
\begin{equation*}
\Gamma_{L,c}(K,r,\varepsilon,g,N):=N+c\left(\frac{\varepsilon}{L(K+r)}\right)
\end{equation*}
where now we only require that $L>\mu_0$. One can easily show that in this case $\Gamma^\ast_{L,c}(K,r,\varepsilon,g,N)=\Gamma_{L,c}(K,r,\varepsilon,g,N)$, and therefore in this case, by Theorem \ref{res-mined}, a rate of metastability for the Picard iterates would be given by
\begin{equation*}
\Omega_{L,c}(\Phi,K,r,\varepsilon,g):=(g_\omega)^{(\ceil{K/\eta})}(0)+\omega
\end{equation*}
where $g_\omega(j):=j+\omega+g^\ast(j+\omega)$ and $\omega:=c\left(\frac{\eta}{L(K+r)}\right)$, and $\eta$ is again defined as in Theorem \ref{res-mined}.\\

\noindent\textbf{Acknowledgements.} I am grateful to Ulrich Kohlenbach for proposing the study of \cite{KirSim(1999.0)}, and to both Ulrich Kohlenbach and Andrei Sipo\c s for numerous helpful comments and suggestions on earlier drafts of this paper. I would also like to thank the anonymous referee for their corrections. This work was supported by the German Science Foundation (DFG Project KO 1737/6-1).

\bibliographystyle{plain}
\bibliography{/home/thomas/Dropbox/tp}

\end{document}